\documentclass[3p,times]{elsarticle}
\usepackage{amsmath,amssymb,amsthm}
\usepackage{xcolor}
\usepackage{graphicx}
\usepackage{hyperref}
\usepackage{booktabs}
\input diagxy

\hypersetup{colorlinks=true,linkcolor=blue,citecolor=blue,urlcolor=blue}

\theoremstyle{plain}
\newtheorem{theorem}{Theorem}[section]
\newtheorem{thm}[theorem]{Theorem}
\newtheorem{prop}[theorem]{Proposition}
\newtheorem{lem}[theorem]{Lemma}

\theoremstyle{definition}
\newtheorem{defn}[theorem]{Definition}
\newtheorem{rem}[theorem]{Remark}
\newtheorem{exmp}[theorem]{Example}

\newcommand{\eps}{\varepsilon}

\newcommand{\lra}{\to}
 \newcommand{\ra}{\rightarrow}
\newcommand{\Dt}{\Delta^{+}}

\newcommand{\TD}{T_{\!D}}

\newcommand{\TNM}{T^{\mathrm{nM}}}

\newcommand{\To}{T_{\! 1}}
\newcommand{\Tt}{T_{\! 2}}
\newcommand{\CE}{\mathcal{E}}
\newcommand{\M}{\mathbb{M}}
\newcommand{\I}{\mathbb{I}}
\newcommand{\lam}{\lambda}

\newcommand{\lan}{\langle}
\newcommand{\ran}{\rangle}

\newcommand{\Tot}{{\To\!\star\!\Tt}}
\newcommand{\DN}{{\TD\!\star\TNM\!}}

\begin{document}

\begin{frontmatter}
\title{Beyond left-continuity: t-norms on the fuzzy unit interval induced by pairs of t-norms}
\author[label1]{Rongjun Wu} 
\ead{wurj@swun.edu.cn, eugen\_woo@163.com}
\author[label1]{Jie Sun} 
\ead{jiesun1027@163.com}
\author[label2]{Hongliang Lai} 
\ead{hllai@scu.edu.cn}
\affiliation[label1]{organization={School of Mathematics,  Southwest Minzu University},
            city={Chengdu},
            postcode={610041}, 
            country={China}}
\affiliation[label2]{organization={School of Mathematics,  Sichuan University},
            city={Chengdu},
            postcode={610064}, 
            country={China}}
\cortext[cor1]{J. Sun is the corresponding author.}
\date{ }
 \begin{abstract}
We define a binary operation $T_1\star T_2$ on the fuzzy unit interval $\I$ by extending the natural action of $T_1$ and $T_2$ on step functions. 
When both $T_1$ and $T_2$ are left-continuous, this operation coincides with the tensor product of the corresponding quantales and therefore yields a t-norm on $\I$.  
Without left-continuity, the situation becomes more subtle. We prove that if at least one of the two t-norms is continuous, then $T_1\star T_2$ is still a t-norm on $\I$.  
We also show by counterexample that continuity cannot simply be dropped. The operation $\DN$ fails to be associative.  
On the other hand, $T_D\star T_D$ is a t-norm despite $T_D$ being neither continuous nor left-continuous, indicating that continuity is not a necessary condition. 
\end{abstract}

\end{frontmatter}

 \section{Introduction}

Triangular norms,  or t-norms, are associative, commutative, non-decreasing binary operations on the real unit interval $[0,1]$ with neutral element $1$.
They were introduced in 1963 by Schweizer and Sklar \cite{Schweizer1963,Schweizer2005} to formulate the triangle inequality in probabilistic metric spaces.
In fuzzy set theory, t-norms have been extensively used to model fuzzy intersections \cite{Alsina1980,Prade1980} and conjunction operators in fuzzy logic \cite{Dubois1980, Hajek1998}.

In 1967, Goguen introduced $L$-fuzzy sets \cite{Goguen1967}, where the codomain of a fuzzy set is a partially ordered set (poset) $L$.
Noting that the definition of a t-norm on $[0,1]$ can be naturally extended to any bounded poset, De Cooman and Kerre generalized the concept of t-norms to bounded posets \cite{Cooman1994}, thereby providing fuzzy intersection operators for Goguen's $L$-fuzzy sets.

Following Goguen's general framework, many different types of posets $L$ have been employed as codomains of fuzzy sets.
Among these, an important class consists of posets into which the unit interval $[0,1]$ can be embedded as a subset, because such posets allow a more precise description of concepts with subtle fuzzy features.
A prominent example is Zadeh's type-2 fuzzy sets \cite{Zadeh1975a,Zadeh1975b}, where the codomain is taken to be the set $\M = [0,1]^{[0,1]}$ of all fuzzy subsets of $[0,1]$.
To extend the usual order on $[0,1]$ to $\M$, two compatible partial orders $\sqsubseteq_{\sqcap}$ and $\sqsubseteq_{\sqcup}$ were introduced on $\M$, yielding the posets $(\M, \sqsubseteq_{\sqcap})$ and $(\M, \sqsubseteq_{\sqcup})$ \cite{Harding2016}.
Also in \cite{Harding2016},  the authors lifted t-norms on $[0,1]$ to binary operations on these two posets by using the meet operation $\wedge$ on the codomain $[0,1]$.

In fuzzy topology, Hutton introduced the fuzzy unit interval $\I(L)$ in 1975 as a lattice-theoretic extension of the real unit interval, where $L$ is a complete lattice \cite{Hutton1975}.
In the special case \(L = [0,1]\), the fuzzy unit interval can be characterized equivalently as the collection of all right-continuous non-increasing functions \(f\colon [0,1]\lra [0,1]\) satisfying \(f(1)=0\), which forms a subset of \(\mathbb{M}\). 
We denote this particular fuzzy unit interval by \(\mathbb{I}\).
Endowed with the pointwise order, \(\mathbb{I}\) is a subposet of \(\mathbb{M}\) equipped with the pointwise order. 

In fact, $\I$ enjoys even stronger structural properties.
The unit interval $[0,1]$ embeds into $\I$ as a complete sublattice, and $\I$ is precisely the tensor product of $[0,1]$ with itself in the category $\mathbf{SUP}$, whose objects are complete lattices and whose morphisms are supremum-preserving maps.
Consequently, t-norms on $\I$ can be constructed by forming the tensor product of two quantales with the same underlying set $[0,1]$ \cite{Eklund2018,Gutierrez2017,Gutierrez2019}.
Recall that $([0,1],T)$ is a quantale if and only if $T$ is left-continuous.
Thus, the quantale tensor product construction of a t-norm on $\I$ requires both input t-norms to be left-continuous.
This observation leads to a natural question: given two arbitrary t-norms $T_1$ and $T_2$ on $[0,1]$ that are not necessarily left-continuous, can one still construct a t-norm on $\I$ from them?

In this paper, we give sufficient conditions under which this construction yields a t-norm. 
We present a general method for constructing a binary operation $T_1\star T_2$ on $\I$ starting from two arbitrary t-norms $T_1$ and $T_2$ on $[0,1]$.
When both t-norms are left-continuous, this operation coincides with the quantale tensor product construction already known in the literature; the existence of the resulting unital quantale follows from \cite[Theorem~4.2]{Gutierrez2017}. Our proof in this case is a direct verification for the unit interval, rather than a new tensor-product construction.
We then prove that $T_1\star T_2$ is a t-norm on $\I$ whenever at least one of $T_1, T_2$ is continuous.
We also provide examples illustrating two complementary phenomena:
on the one hand, if neither component t-norm is continuous, $T_1\star T_2$ need not be associative and may therefore fail to be a t-norm on $\I$;
on the other hand, although the drastic product t-norm $T_D$ is neither continuous nor left-continuous, $T_D\star T_D$ is still a t-norm on $\I$.

The main associativity theorem goes beyond the case of two left-continuous components because the second component may fail left-continuity when the first is continuous. The pair $(T_D,T_D)$ shows that both components may fail left-continuity, whereas $(T_D,T^{\mathrm{nM}})$ shows that left-continuity of only one component does not suffice. These results provide sufficient conditions and contrasting examples, not a classification of all admissible pairs.

The paper is organized as follows.
Section~2 recalls basic notions concerning fuzzy unit intervals and t-norms on $[0,1]$.
In Section~3 we introduce the operation $T_1\star T_2$, establish its fundamental properties, and discuss its Archimedean-type properties.
Section~4 is devoted to proving our main result: if either $T_1$ or $T_2$ is continuous, then $T_1\star T_2$ is a t-norm on $\I$; two important examples are also discussed there.
Section~5 concludes the paper and outlines directions for future work.

\section{Preliminaries}\label{prel}

 Unless explicitly stated otherwise, $(P,\leq)$ denotes a bounded partially ordered set. Its bottom and top elements are denoted by $\bot$ and $\top$, respectively. We always assume that $\bot\neq\top$.

 Equipped with the usual order of real numbers, the unit interval $[0,1]$ is a bounded partially ordered set with $\bot=0$  and $\top=1$. It is also a complete lattice. For a subset $A\subseteq [0,1]$, we write $\sup A$ and $\inf A$ for the least upper bound and greatest lower bound of $A$, respectively. In particular, $\sup\emptyset=0$ and $\inf\emptyset=1$.

 Let $f:[0,1]\lra[0,1]$ be a non-increasing function. Define 
 $$f(t^-)=\inf\limits_{s<t}f(s), \quad f(t^+)=\sup\limits_{s>t}f(s)$$ 
 for all $t\in[0,1]$. In particular, 
 $$f(0^-)=\inf\emptyset=1, \quad f(1^+)=\sup\emptyset=0.$$

Hutton's fuzzy unit interval \cite{Hutton1975} is defined by taking a general complete lattice $L$ as the codomain of the fuzzy sets. In the present paper we restrict to the case $L=[0,1]$.  

\begin{defn} 
Let $\M^R$ be the set of all non-increasing maps $\lam:[0,1]\lra [0,1]$ with $\lam(0)=1$ and $\lam(1)=0$.
Let $\sim$ be an equivalence relation on $\M^R$ given by
$$
\lam\sim\mu\iff \forall t\in[0,1],\, 
\lam(t^-)=\mu(t^-),\, \lam(t^+) = \mu(t^+). 
$$
The quotient set $\M^R/\sim$ is called the fuzzy unit interval.
\end{defn}
 
We denote by \(\mathbb{I}\) the collection of all right-continuous non-increasing maps \(f\colon [0,1]\to [0,1]\) satisfying \(f(1)=0\). Then a map \(f\colon [0,1]\to [0,1]\) belongs to \(\mathbb{I}\) if and only if
\(f(\inf A) = \sup f(A)\) for all \(A\subseteq [0,1]\).

 \begin{prop}
There is a bijection between the set $\I$ and the quotient set $\M^R/\sim$.
\end{prop}

\begin{proof} Given an element $\lam\in \M^R$, let $\bar{\lam}(t)=\lam(t^+)$ for all $t\in[0,1]$. Then $\bar{\lam}\in\I$. 
We first show that the map $[\lambda] \mapsto \bar{\lambda}$ is well-defined from $\M^R/\sim$ to $\I$. By the definition of the equivalence relation $\sim$ on $\M^{R}$, if $\lambda \sim \mu$, then clearly $\bar{\lambda} = \bar{\mu}$. Thus, the assignment is independent of the choice of representatives.

  To prove injectivity, assume $[\lambda] \neq [\mu]$. By the definition of $\sim$, either $\lambda(t_0^{-}) \neq \mu(t_0^{-})$ or $\lambda(t_0^{+}) \neq \mu(t_0^{+})$ for some $t_0 \in [0,1]$. We consider two cases:
  \begin{enumerate}[(i)]
    \item If $\lambda(t_0^{+}) \neq \mu(t_0^{+})$ for some $t_0$, then clearly $\bar{\lambda} \neq \bar{\mu}$.
    \item If $\lambda(t_0^{-}) \neq \mu(t_0^{-})$ for some $t_0$, without loss of generality assume $\lambda(t_0^{-}) < \mu(t_0^{-})$. Choose $\delta > 0$ such that for all $t \in (t_0 - \delta, t_0)$,
    $$
    \lambda(t) < \frac{\lambda(t_0^{-}) + \mu(t_0^{-})}{2} < \mu(t).
    $$
 For all $t \in (t_0 - \delta, t_0)$,  since  $\mu(t)$ is  greater than  the constant  $\frac{\lambda(t_0^{-}) + \mu(t_0^{-})}{2} $,  we have  $\frac{\lambda(t_0^{-}) + \mu(t_0^{-})}{2} \leq \mu(t^{+}).$
    Consequently,  for all $t \in (t_0 - \delta, t_0)$, 
    $$
    \lambda(t^{+}) \leq \lambda(t) < \frac{\lambda(t_0^{-}) + \mu(t_0^{-})}{2} \leq \mu(t^{+}),
    $$
    which implies $\bar{\lambda} \neq \bar{\mu}$.
  \end{enumerate}
  Thus, the map is injective.

  For surjectivity, let $\lambda \in \I$ be arbitrary. Define $\tilde{\lambda} \in \M^R$ by
  $\tilde{\lambda}(t) = 
  \begin{cases}
    \lambda(t), & t > 0, \\
    1, & t = 0.
  \end{cases}
  $
  Then $\bar{\tilde{\lambda}} = \lambda$, showing that the map is also surjective.

  Therefore, the assignment $[\lambda] \mapsto \bar{\lambda}$ is a bijection as desired.
\end{proof}

\begin{defn}\label{def:adjoint}
For $f\in\I$, define its \emph{adjoint} by
\[
 f^*(y)=\inf\{x\in[0,1]:f(x)\le y\},\qquad y\in[0,1].
\]
The set under the infimum is nonempty because $f(1)=0$.
\end{defn}

\begin{lem}\label{lem:adjoint}
For $f\in\I$, the map $f^*$ is the unique map satisfying
\[
 f(x)\le y\quad\Longleftrightarrow\quad f^*(y)\le x
 \qquad(x,y\in[0,1]).
\]
Moreover, $f^*\in\I$, $(f^*)^*=f$, and
$f\le g$ if and only if $f^*\le g^*$.
\end{lem}
\begin{proof}
Fix $y$ and put $A_y=\{x:f(x)\le y\}$.
Since $f$ takes arbitrary infima to suprema,
$f(\inf A_y)=\sup f(A_y)\le y$.
Thus $A_y=[f^*(y),1]$, proving the equivalence and uniqueness.
For every $B\subseteq[0,1]$ and $x\in[0,1]$, this equivalence gives
\[
\begin{aligned}
 f^*(\inf B)\le x
 &\Longleftrightarrow f(x)\le\inf B\\
 &\Longleftrightarrow (\forall b\in B)\ f^*(b)\le x\\
 &\Longleftrightarrow \sup_{b\in B}f^*(b)\le x.
\end{aligned}
\]
Consequently $f^*(\inf B)=\sup_{b\in B}f^*(b)$, including
$B=\varnothing$, and therefore $f^*\in\I$.
The symmetry of the adjunction equivalence gives $(f^*)^*=f$.
If $f\le g$, then $\{x:g(x)\le y\}\subseteq\{x:f(x)\le y\}$,
so $f^*\le g^*$. Applying this implication to $f^*,g^*$ proves
the converse.
\end{proof}
Thus $f\mapsto f^*$ is an order-preserving involution, hence a
complete lattice automorphism.

Step functions play a fundamental role in the structure of $\I$. For $r,a\in[0,1]$, define the \emph{step function} $\lan r,a\ran:[0,1]\lra[0,1]$ by
$$
\lan r,a\ran(x)=
\begin{cases}
a,& 0\le x<r,\\
0,& r\le x\le 1.
\end{cases}
$$

Equipped with the pointwise order, $\I$ is a complete lattice whose top and bottom elements are $\lan 1,1\ran$ and $\lan 0,0\ran$, respectively. Suprema are calculated pointwise: for every family $(f_i)_{i\in J}\subseteq\I$,
\[
 \left(\sup_{i\in J}f_i\right)(x)=\sup_{i\in J}f_i(x).
\]
Indeed, the right-hand side is non-increasing, vanishes at $1$, and satisfies, for $x<1$,
\[
 \sup_{y>x}\sup_{i\in J}f_i(y)
 =\sup_{i\in J}\sup_{y>x}f_i(y)
 =\sup_{i\in J}f_i(x).
\]
It therefore belongs to $\I$ and is the least upper bound. This also covers the empty family.
Let
$$
\CE:=\{\lan r,a\ran\mid r,a\in[0,1]\}.
$$
The following lemma shows that $\CE$ is sup-dense in $\I$.

\begin{lem}\label{lem:step-supdense} 

(1) For all $f\in\I$ and all $\lan s,a\ran\in\CE$, it holds that 
\begin{equation}\label{eq:step-order-char}
\lan s,a\ran\le f
\iff
a\le f(s^-).
\end{equation}

(2)
For each $f\in\I$, it holds that 
\begin{equation}
f=\sup_{r\in[0,1]}\lan r,f(r)\ran
=\sup_{r\in(0,1]}\lan r,f(r^-)\ran.
\end{equation}
Consequently, for every $x\in[0,1]$,
$$
f(x)=\sup_{x<r}f(r^-).
$$

\end{lem}

\begin{proof}
(1) 
We have 
\begin{align*}
\lan s,a\ran\leq f&\iff \forall  x<s, \lan s,a\ran(x)\leq f(x)\\
&\iff \forall x<s, a\leq f(x)\\
&\iff a\leq f(s^-).
\end{align*}

(2) First, for every $x\in[0,1]$, we have 
\begin{align*}
f(x)&=\sup_{r>x}f(r)\quad\quad  (\text{since }\forall A\subseteq [0,1],f(\inf A)=\sup f(A))\\
&=\sup_{r>x}\lan r,f(r)\ran (x) \quad\quad (\text{by the definition of } \lan r,f(r)\ran)\\
&=\sup_{r\in[0,1]}\lan r,f(r)\ran (x).\quad\quad (\text{since }\lan r,f(r)\ran (x)=0 \text{ if } r\leq x)
\end{align*}
Hence, $$f=\sup\limits_{r\in[0,1]}\lan r,f(r)\ran.$$

Second, notice that $f(r)\leq f(r^-)$ since $f$ is non-increasing. By (1), it holds that $$\lan r,f(r)\ran\leq \lan r,f(r^-)\ran\leq f.$$   Hence,  
$$f=\sup_{r\in [0,1]}\lan r,f(r)\ran\le\sup_{r\in(0,1]}\lan r,f(r^-)\ran\le f,$$  so \(f=\sup_{r\in(0,1]}\lan r,f(r^-)\ran\). By the definition of step function,    $$\lan r,f(r^-)\ran (x)=\begin{cases}0, & x\ge r,\\
f(r^-), &  x<r.
\end{cases}$$ Thus, it holds that
$$f(x)=\sup_{r\in(0,1]}\lan r,f(r^-)\ran (x)=\sup_{x<r}\lan r,f(r^-)\ran (x)=\sup_{x<r}f(r^-).$$
\end{proof}

\begin{rem}\label{rem:unibimor}
A map from the product of two complete lattices to a complete
lattice is a \emph{bimorphism} if it
preserves arbitrary suprema in each variable separately, including
the empty supremum. The map
\[
 \eps:[0,1]\times[0,1]\longrightarrow\I,
 \qquad \eps(r,a)=\lan r,a\ran,
\]
is a universal bimorphism. To see this directly, the pointwise
join formula shows that $\eps$ preserves arbitrary suprema in
each variable. Given a bimorphism
$\varphi:[0,1]\times[0,1]\to L$, where $L$ is a complete lattice,
define
\[
 \bar\varphi(f)=\sup_{r\in[0,1]}\varphi(r,f(r)).
\]
For any family $(f_i)_{i\in J}$, preservation of suprema by
$\varphi$ in its second variable gives
\[
 \bar\varphi\left(\sup_i f_i\right)
 =\sup_r\varphi\left(r,\sup_i f_i(r)\right)
 =\sup_i\sup_r\varphi(r,f_i(r))
 =\sup_i\bar\varphi(f_i).
\]
Moreover, $\varphi(r,0)=\bot_L$ and preservation in the first
variable yield
\[
 \bar\varphi(\lan s,a\ran)
 =\sup_{r<s}\varphi(r,a)=\varphi(s,a),
\]
where the case $s=0$ uses the empty supremum.
Hence $\bar\varphi\circ\eps=\varphi$.
Uniqueness follows from
$f=\sup_r\lan r,f(r)\ran$ and preservation of suprema.
This proves that $\I$, together with $\eps$, realizes the tensor
product of $[0,1]$ with itself in $\mathbf{SUP}$
\cite{Banaschewski1976}.
Let $b:[0,1]\times[0,1]\to[0,1]\otimes[0,1]$ be the canonical universal bimorphism in a standard tensor-product model, such as those described via Galois connections or lower sets in \cite{Gutierrez2017,Shmuely1974}. The universal properties give a unique complete lattice isomorphism
\[
 U:[0,1]\otimes[0,1]\longrightarrow\I,
 \qquad U(b(r,a))=\lan r,a\ran.
\]
Thus $\eps=U\circ b$: our map is the canonical universal bimorphism expressed in the step-function realization of the same tensor product.
The map $\eps$ is not injective, since
$\lan r,0\ran=\lan0,a\ran=\bot$ for all $r,a$.
Thus $\eps$ is not an embedding of the product poset $[0,1]\times[0,1]$.
\end{rem}

\begin{rem}\label{IMD}
Since there exists a bijection between the set $\M^R/\sim$ and $\I$, we often take $\I$ to represent the fuzzy unit interval. In fact, adopting the set $\I$ to denote the fuzzy unit interval offers some advantages. Specifically, it allows us to more explicitly reveal its close relationships with the following two mathematical objects.  

\begin{enumerate}[(1)]
\item $\I$ is a subset of the codomain of type-2 fuzzy sets. Recall from \cite{Harding2016,Zadeh1975a,Zadeh1975b} that the set $\M=[0,1]^{[0,1]}$, consisting of all fuzzy subsets of $[0,1]$, is the codomain of type-2 fuzzy sets  introduced by  Zadeh.  Clearly, $\I$ is a subset of $\M$. 
Moreover, equipped with the pointwise order on $\I$, the inclusion map $\I\hookrightarrow \M$ is also an order embedding with respect to both  orders $\sqsubseteq_{\sqcap}$ and $\sqsubseteq_{\sqcup}$. 

\item $\I$ is isomorphic to $\Dt$. Recall from \cite{Schweizer2005} that  $\Dt$ is the set of all distance distribution functions. That is, for all $F\in\Dt$, $F$ is a non-decreasing map from $[0,\infty]$ to $[0,1]$ such that   $F(0)=0$, $F(\infty)=1$ and $F$  is left-continuous on  $(0,\infty)$. Notice that $x\mapsto -\ln x$ is an isomorphism between $([0,1],\leq)$ and $([0,\infty],\geq)$. This order isomorphism gives rise to an isomorphism $$\Phi:\Dt\lra\I,\quad   
 \Phi(F)(x)=\begin{cases}F(-\ln x),& 0<x\leq 1,\\
                         F(\infty^-), &   x=0.
            \end{cases}
$$
\end{enumerate}
\end{rem}

\begin{defn}\cite{Cooman1994}
Let $(P,\leq)$ be a bounded partially ordered set. A \emph{triangular norm}, or \emph{t-norm}, on $(P,\leq)$ is a binary operation
$
T:P\times P\lra P
$
satisfying the following conditions: 
\begin{enumerate}[(T1)] 
\item  $\forall x\in P, T(x,\top)=T(\top,x)=x$,
\item  $\forall x,y\in P, T(x,y)=T(y,x)$,
\item  $\forall x,y,z\in P, T(T(x,y),z)=T(x,T(y,z))$,
\item  $\forall x,x',y,y'\in P, x\le x', y\le y'\implies T(x,y)\le T(x',y')$.
\end{enumerate}
\end{defn}
There are some important basic t-norms on the interval $[0,1]$; see, e.g., \cite{Alsina2006,Klement2000}. We list them below:
\begin{itemize}
\item minimum: $T_M(x,y)=\min\{x,y\}$,
\item product: $T_P(x,y)=x\cdot y$,
\item {\L}ukasiewicz t-norm: $T_{L}(x,y)=\max\{x+y-1,0\}$,
\item drastic product: $
               \TD(x,y)=
               \begin{cases}
               \min\{x,y\}, & \max\{x,y\}=1,\\
               0, & \max\{x,y\}<1,
              \end{cases}
                   $
\item the nilpotent minimum t-norm:
$
\TNM(x,y)=
\begin{cases}
\min\{x,y\}, & x+y>1,\\
0, & x+y\le 1.
\end{cases}
$
\end{itemize}

The t-norms $T_M$, $T_P$ and $T_L$ are continuous. The nilpotent minimum t-norm $\TNM$ is left-continuous but not continuous, whereas the drastic product $T_D$ is not left-continuous. 
  
\section{Operations on the fuzzy unit interval induced by pairs of t-norms}\label{sec3}

Given t-norms $\To$ and $\Tt$ on $[0,1]$, we obtain a binary operation
\begin{equation}\label{times}\Tot:\CE\times \CE\lra \CE,\quad \Tot(\lan r,a\ran,\lan s,b\ran )=\lan \To(r,s),\Tt(a,b)\ran.\end{equation}
This prescription is independent of the representation of a step
function. Indeed, every nonzero step has unique parameters
$r,a>0$, whereas $\lan r,a\ran=\bot$ precisely when $r=0$ or
$a=0$. Since every t-norm satisfies $T(0,t)=0$, a zero input
always gives the zero step, independently of its parameters.
For nonzero steps, $\lan r,a\ran\le\lan s,b\ran$ is equivalent
to $r\le s$ and $a\le b$. Consequently the operation on $\CE$
is monotone. Its commutativity, associativity, and neutral element
$\lan1,1\ran$ follow from the corresponding properties of
$T_1,T_2$. Since $\CE$ is a sup-dense subset of $\I$, the operation $\Tot$ can be naturally extended to $\I$ as follows:
\begin{align*}
\Tot(f,g)
&=\sup\{\Tot(\lan s,a\ran,\lan t,b\ran)\mid \lan s,a\ran\leq f,\ \lan t,b\ran\leq g\}\\
&=\sup\{\lan \To(s,t),\Tt(a,b)\ran \mid \lan s,a\ran\leq f,\ \lan t,b\ran\leq g\}.
\end{align*}

By monotonicity, when both inputs are steps, every term in the supremum is bounded by their product in \eqref{times}, and that product itself is an admissible term. Thus the extension agrees with \eqref{times} on $\CE$.

\begin{prop} \label{tens C}
Let $\To$ and $\Tt$ be t-norms on $[0,1]$. Then  
\begin{enumerate}[(1)]
\item $\Tot(f,g)=\Tot(g,f)$;
\item If $f_1\le f_2$ and $g_1\le g_2$, then $\Tot(f_1,g_1)\le{\Tot}(f_2,g_2)$;
\item ${\Tot}(f,\lan 1,1\ran)=f$.
\end{enumerate}
\end{prop}

\begin{proof}
 
(1) Since both $T_1$ and $T_2$ are commutative, the commutativity of $\Tot$ follows directly from its definition. 
   
(2) Assume that $f_1\le f_2$. 
 For each step function $\lan r,a\ran\in\CE$, if $\lan r,a\ran\leq f_1$ then $\lan r,a\ran\leq f_2$. 
Therefore, it follows that
$$
\Tot(f_1,g_1)\le \Tot(f_2,g_2)
$$
by the definition of $\Tot$ straightforwardly.

(3)Since $\lan 1,1\ran$ is the top in $\I$, we have
\begin{align*}
\Tot(f,\lan 1,1\ran)
&=\sup\{\lan \To(r,t),\Tt(a,b)\ran
\mid\lan r,a\ran\le f,\ \lan t,b\ran\leq\lan 1,1\ran\}\\
&=\sup\{\lan \To(r,1),\Tt(a,1)\ran
\mid \lan r,a\ran\leq f\}\\
&=\sup\{\lan r,a\ran
\mid a\le f(r^-)\}\\
&=\sup_{r\in(0,1]}\lan r,f(r^-)\ran\\
&=f.
\end{align*}
The second equality follows from $t,b\le1$, the monotonicity of $\To,\Tt$, and the fact that $t=b=1$ is admissible. The last equality is Lemma~\ref{lem:step-supdense}(2).
\end{proof}

\begin{prop}\label{tens A}
Let $\To$ and $\Tt$ be t-norms on $[0,1]$. Then, for all $f,g\in\I$,
$$
(\Tot(f,g))^*=\Tt\star\To(f^*,g^*).
$$
Consequently, $f\mapsto f^*$ is an order isomorphism $(\I,\Tot)\cong(\I,\Tt\star\To)$ that intertwines the two operations.
\end{prop}
\begin{proof}
Notice that for every step function $\lan r,a\ran\in\CE$, one has
$$
\lan r,a\ran^*=\lan a,r\ran.
$$
Consequently, for every $f\in\I$,
$$
\lan r,a\ran\le f
\iff
\lan a,r\ran\le f^*.
$$

By Lemma~\ref{lem:adjoint}, the map $f\mapsto f^*$ is an order-preserving involution, hence an order isomorphism of the complete lattice $\I$. It therefore preserves all suprema. Hence 
\begin{align*}
(\Tot(f,g))^*
&=
\left(
\sup\bigl\{
\lan \To(r,s),\Tt(a,b)\ran
\mid
\lan r,a\ran\le f,\ \lan s,b\ran\le g
\bigr\}
\right)^*\\
&=
\sup\bigl\{
\lan \To(r,s),\Tt(a,b)\ran^*
\mid
\lan r,a\ran\le f,\ \lan s,b\ran\le g
\bigr\}\\
&=
\sup\bigl\{
\lan \Tt(a,b),\To(r,s)\ran
\mid
\lan r,a\ran\le f,\ \lan s,b\ran\le g
\bigr\}.
\end{align*} 
By the above equivalence,
$$
\lan r,a\ran\le f \iff \lan a,r\ran\le f^*,
\qquad
\lan s,b\ran\le g \iff \lan b,s\ran\le g^*.
$$
Therefore,
\begin{align*}
(\Tot(f,g))^*
&=
\sup\bigl\{
\lan \Tt(a,b),\To(r,s)\ran
\mid
\lan a,r\ran\le f^*,\ \lan b,s\ran\le g^*
\bigr\}\\
&=
\Tt\star\To(f^*,g^*).
\end{align*}
Thus,
$$
(\Tot(f,g))^*=\Tt\star\To(f^*,g^*).
$$
Thus $f\mapsto f^*$ yields the asserted isomorphism.
\end{proof}
\begin{prop}\label{tens B}
Let $\To$ and $\Tt$ be t-norms on $[0,1]$. Then, for all $f,g\in\I$ and all $x\in[0,1]$,
\begin{equation*}
\Tot(f,g)(x)=
\sup_{x<\To(s,t)} \Tt\bigl(f(s^-),g(t^-)\bigr).
\end{equation*}
\end{prop}

\begin{proof}
By the definition of $\Tot$ and Lemma \ref{lem:step-supdense}(1), we have
\begin{align*}
\Tot(f,g)&=\sup\{\lan \To(s,t),\Tt(a,b)\ran
\mid\lan s,a\ran\le f,\ \lan t,b\ran\le g\}\\
&=\sup\{\lan \To(s,t),\Tt(f(s^-),g(t^-))\ran\mid s,t \in[0,1]\}.
\end{align*} 
 
Now for all $x\in[0,1]$ and all $c\in[0,1]$, it holds that
\begin{align*}
 c<\Tot(f,g)(x)
&\iff \exists r,s\in[0,1], c<\lan \To(r,s),\Tt(f(r^-),g(s^-))\ran(x)\\
&\iff \exists r,s\in[0,1], x<\To(r,s), c<\Tt(f(r^-),g(s^-))\\
&\iff c<\sup_{x<\To(r,s)}\Tt(f(r^-),g(s^-)).
\end{align*}
Therefore, it follows that 
 $$\Tot(f,g)(x)=\sup_{x<\To(r,s)}\Tt(f(r^-),g(s^-))$$ for all $x\in[0,1]$.
\end{proof}

If both $\To$ and $\Tt$ are left-continuous t-norms, then both $([0,1],T_1)$ and $([0,1],T_2)$ are quantales. Recall from \cite{Eklund2018,Gutierrez2017,Gutierrez2019} that the monoidal operation $T_1\otimes T_2$ of the tensor product $([0,1],T_1)\otimes([0,1],T_2)$ is given by  
$$
\To\otimes\Tt(f,g)=\sup_{r,s\in[0,1]}\lan \To(r,s),\Tt(f(r),g(s))\ran.  
$$
Particularly, for each point $x\in[0,1]$, 
$$T_1\otimes T_2(f,g)(x)=\sup_{x<T_1(r,s)}T_2(f(r),g(s)).$$

\begin{thm}\label{tens prod} 
Let both $\To$ and $\Tt$ be left-continuous t-norms on $[0,1]$. Then it holds that $\Tot=T_1\otimes T_2$.
\end{thm}
 
\begin{proof} 
On the one hand, by Proposition~\ref{tens B}, we obtain
$$
\Tot(f,g)(x)\geq\To\otimes\Tt(f,g)(x).
$$
On the other hand, for any $p\in[0,1]$, if $p<\Tot(f,g)(x)$, then there exists a pair $(r,s)$ with $\To(r,s)>x$ such that $p<\Tt(f(r^-),g(s^-))$. Since $\To$ is left-continuous, there is a pair $(r',s')$ satisfying $r'<r$, $s'<s$ and $x<\To(r',s')\leq\To(r,s)$. Thus,
$$
p<\Tt(f(r^-),g(s^-))\leq\Tt(f(r'),g(s'))\leq \sup_{x<T_1(r',s')}T_2(f(r'),g(s')) =\To\otimes\Tt(f,g)(x).
$$
Since $p$ is arbitrary, we conclude that
$$
\Tot(f,g)(x)\leq \To\otimes\Tt(f,g)(x).
$$
Therefore, we have shown that $\Tot=\To\otimes\Tt$, as desired.
\end{proof}

\begin{rem}\label{rem:comparison}
(1) Assume that both $\To$ and $\Tt$ are left-continuous.
Notice that the map $x\mapsto\mathrm e^{-x}$ gives an order
isomorphism from $([0,\infty],\geq)$ to $([0,1],\leq)$.
Hence $\To$ induces a right-continuous t-conorm $L$ on
$[0,\infty]$ by
\[
 L(x,y)=-\ln\bigl(\To(\mathrm e^{-x},\mathrm e^{-y})\bigr),
\]
where $-\ln0=\infty$ and $\mathrm e^{-\infty}=0$.
For $r\in[0,\infty]$ and $p\in[0,1]$, let $\eta_{r,p}$ denote
the step function written as $\varepsilon(r,p)$ in
\cite[Example~2.7]{Lai2026}; explicitly,
\[
 \eta_{r,p}(x)=
 \begin{cases}
  1,&x=\infty,\\
  p,&r<x<\infty,\\
  0,&0\le x\le r,\quad x<\infty.
 \end{cases}
\]
Recall from Remark~\ref{IMD}(2) that
$\Phi:\Dt\longrightarrow\I$ is the order isomorphism induced
by the logarithmic change of variables. Under this identification,
$\Phi(\eta_{r,p})=\langle\mathrm e^{-r},p\rangle$.
Therefore the tensor product $L\otimes\Tt$ on $\Dt$ is
transported to the tensor product $\To\otimes\Tt$ on $\I$;
namely,
\[
 \Phi\bigl((L\otimes\Tt)(F,G)\bigr)
 =(\To\otimes\Tt)\bigl(\Phi(F),\Phi(G)\bigr),
 \qquad F,G\in\Dt.
\]
Consequently, Theorem~\ref{tens prod} gives
$\Tot=\To\otimes\Tt\cong L\otimes\Tt$.
Thus, in the left-continuous case, our operation is precisely
the logarithmic counterpart of the tensor-product triangle
function investigated in \cite{Lai2026}.
If, in addition, $\To$ is continuous, then $L$ is continuous.
In this case, \cite[Theorem~3.18]{Lai2026} implies that
$\tau_{\Tt,L}=L\otimes\Tt$ if and only if $L$ satisfies
the conditionally strictly increasing property
\[
 x<x',\quad y<y',\quad L(x',y')<\infty
 \quad\Longrightarrow\quad L(x,y)<L(x',y')
 \qquad\text{(LCS)}
\]
for all $x,x',y,y'\in[0,\infty]$.
Under this condition, $\Tot$ also corresponds to the classical
Schweizer--Sklar triangle function $\tau_{\Tt,L}$.

(2) Recall from \cite{Harding2016} that the convolution of a
t-norm $T$ is the operation
$\mathcal T:\M\times\M\longrightarrow\M$ defined by
\[
 \mathcal T(f,g)(x)=\sup_{T(r,s)=x}\min\{f(r),g(s)\}.
\]
Let $T$ be continuous and satisfy
\[
 u<u',\quad v<v'\quad\Longrightarrow\quad
 T(u,v)<T(u',v')
 \qquad(u,u',v,v'\in[0,1]).
\]
Then $T\star T_M(f,g)=\mathcal T(f,g)$ for all $f,g\in\I$.
Indeed, Theorem~\ref{tens prod} gives
\[
 (T\star T_M)(f,g)(x)
 =\sup_{T(r,s)>x}\min\{f(r),g(s)\}.
\]
For $0\le x<1$, continuity of $T$ permits lowering a coordinate
until $T(r,s)=x$, proving that this supremum is at most
$\mathcal T(f,g)(x)$. Conversely, any pair with $T(r,s)=x$
and $r,s<1$ can be approached from above by pairs with
$T(r_n,s_n)>x$; right-continuity of $f,g$ gives the reverse
inequality. Pairs with $r=1$ or $s=1$ contribute zero, and at
$x=1$ both expressions vanish. Consequently, the convolution
leaves $\I$ invariant and coincides there with $T\star T_M$.
\end{rem}

In what follows, we discuss the Archimedean property of $\Tot$.

Let $T$ be a t-norm on $[0,1]$. For each $x\in[0,1]$, define
$x_1=x$ and $x_{n+1}=T(x_n,x)$ for all $n\geq1$.
Recall that $T$ is \emph{Archimedean} if, for all $x,y\in(0,1)$,
there exists $n\geq1$ such that $x_n<y$. Equivalently,
$\inf_{n\geq1}x_n=0$ for every $x\in(0,1)$; see
\cite{Klement2000}.
The definition of Archimedean property of t-norms  on  a  bounded poset were proposed by De Baets and Mesiar as follows. 

\begin{defn}\label{3.6s}\cite{DeBaets1999} A t-norm \(T\) on a bounded poset \((P,\le)\) is called Archimedean if for any \((x, y) \in P^2,\)
\[
(\forall n \ge 1) (x_n \geqslant y) \Rightarrow (x = \top \text{ or } y = \bot),
\]
where $x_1=x$ and $x_{n+1}=T(x_n,x)$ for all $n\geq1$.
\end{defn}

\begin{prop}\label{3.7s}\cite{DeBaets1999} A t-norm \(T\) on a complete lattice \(L\) is Archimedean if and only if
\[
(\forall x \in L \setminus \{\top\}) \left( \inf_{n \ge 1} x_n = \bot \right),
\] where $x_1=x$ and $x_{n+1}=T(x_n,x)$ for all $n\geq1$.
\end{prop}

We first establish a property of $\Tot$ concerning finite suprema.

\begin{lem}\label{lem:archnc-finite-joins}
Let $\To$ and $\Tt$ be t-norms on $[0,1]$. Then $\Tot$
preserves finite suprema in each argument. In particular, for all
$f,g,h\in\I$, it holds that
$$
 \Tot(\sup\{f,g\},h)
 =\sup\{\Tot(f,h),\Tot(g,h)\}.
$$
\end{lem}

\begin{proof}
The supremum of $f$ and $g$ is calculated pointwise. Since both
functions are non-increasing, for every $r\in[0,1]$ we have
$$
 \bigl(\sup\{f,g\}\bigr)(r^-)
 =\max\{f(r^-),g(r^-)\}.
$$
Consequently, by \eqref{eq:step-order-char}, for every
$\lan r,a\ran\in\CE$,
\begin{align*}
 \lan r,a\ran\leq\sup\{f,g\}
 &\iff a\leq\max\{f(r^-),g(r^-)\}\\
 &\iff \lan r,a\ran\leq f\ \text{or}\ \lan r,a\ran\leq g.
\end{align*}
Substituting this equivalence into the definition of $\Tot$ in
terms of step functions gives the required equality.
Commutativity gives the same conclusion for the second argument.
Moreover, $\Tot(\bot,h)=\bot$, so the empty supremum is also
preserved.
\end{proof}

\begin{prop}\label{prop:archnc-inheritance}
Let $\To$ and $\Tt$ be Archimedean t-norms on $[0,1]$.
For each $f\in\I$, define
$$
 f_1=f,\qquad f_{n+1}=\Tot(f_n,f)\quad(n\geq1).
$$
Then, for every $f<\top$, it holds that
\begin{equation}\label{eq:archnc-order-limit}
 \inf_{n\geq1}f_n=\bot,
\end{equation}
where the infimum is taken in the complete lattice $\I$.
If $\Tot$ is associative, then it is an Archimedean t-norm on $\I$ .
\end{prop}

\begin{proof}
Let $f<\top$. There exist $r,a\in(0,1)$ such that $f(r)\leq a$.
Set
$$
 g=\sup\{\lan r,1\ran,\lan1,a\ran\}.
$$
By the monotonicity of $f$ and the definition of step functions,
we have $f\leq g$.
Let $g_1=g$ and $g_{n+1}=\Tot(g_n,g)$ for all $n\geq1$.
Also, let $r_0=a_0=1$ and define
$$
 r_{n+1}=\To(r_n,r),\qquad
 a_{n+1}=\Tt(a_n,a)\quad(n\geq0).
$$
Since $\To$ and $\Tt$ are Archimedean, $r_n\to0$ and $a_n\to0$.

First, we show by induction that
\begin{equation}\label{eq:archnc-finite-expansion}
 g_n=\sup_{0\leq j\leq n}\lan r_j,a_{n-j}\ran
 \quad(n\geq1).
\end{equation}
For $n=1$, the equality follows from the definition of $g$.
Assume that it holds for some $n\geq1$.
By Lemma~\ref{lem:archnc-finite-joins} and the formula for
$\Tot$ on step functions, we obtain
\begin{align*}
 g_{n+1}
 &=\Tot\left(\sup_{0\leq j\leq n}\lan r_j,a_{n-j}\ran,
             \sup\{\lan r,1\ran,\lan1,a\ran\}\right)\\
 &=\sup_{0\leq j\leq n}
     \sup\{\lan r_{j+1},a_{n-j}\ran,
            \lan r_j,a_{n-j+1}\ran\}\\
 &=\sup_{0\leq j\leq n+1}\lan r_j,a_{n+1-j}\ran.
\end{align*}
Thus \eqref{eq:archnc-finite-expansion} holds for every $n\geq1$.

Second, the monotonicity of $\Tot$ gives $f_n\leq g_n$ for
all $n\geq1$. Since the sequences $\{r_n\}$ and $\{a_n\}$
are non-increasing, \eqref{eq:archnc-finite-expansion} yields
\begin{equation}\label{eq:archnc-even-bound}
 f_{2n}\leq g_{2n}
 \leq\sup\{\lan r_n,1\ran,\lan1,a_n\ran\}
 \quad(n\geq1).
\end{equation}
Indeed, for each $j\in\{0,\ldots,2n\}$, either $j\geq n$
or $2n-j\geq n$.

Suppose that $\inf_{k\geq1}f_k\neq\bot$. There exist $u,v\in(0,1]$
such that
$$
 \lan u,v\ran\leq \inf_{k\geq1}f_k.
$$
Choose $n$ such that $r_n<u$ and $a_n<v$, and take
$x\in(r_n,u)$. From \eqref{eq:archnc-even-bound}, we obtain
$$
 v=\lan u,v\ran(x)\leq \inf_{k\geq1}f_k(x)\leq f_{2n}(x)\leq a_n<v,
$$
a contradiction. This proves \eqref{eq:archnc-order-limit}.

If $\Tot$ is associative,  by Proposition \ref{tens C},   it is a t-norm on $\I$.  Hence, $\Tot$ is an Archimedean t-norm on $\I$ from  Proposition \ref{3.7s}.
\end{proof}

\begin{rem}\label{rem:archnc-finite-power}
When $\Tot$ is associative, the sequence $\{f_n\}$ in
Proposition~\ref{prop:archnc-inheritance} consists of the powers
of $f$ with respect to $\Tot$.
It should be pointed out that \eqref{eq:archnc-order-limit}
does not imply the stronger condition
$$
 \forall\,\bot<f,g<\top,\quad
 \exists n\geq1,\quad f_n\leq g.
$$
In fact, take $\To=\Tt=T_P$. Then $\Tot$ is a t-norm by
Theorem~\ref{tens prod}, and both $\To$ and $\Tt$ are
Archimedean. Let
$$
 f=\lan\tfrac12,1\ran,\qquad g=\lan\tfrac12,\tfrac12\ran.
$$
By the formula for $\Tot$ on step functions, we have
$$
 f_n=\lan2^{-n},1\ran\quad(n\geq1).
$$
For every $n\geq1$, taking $x=2^{-n-1}$ gives
$$
 f_n(x)=1>\tfrac12=g(x).
$$
Hence $f_n\nleq g$ for all $n\geq1$, even though
$\bot<g<f<\top$.
\end{rem}

\section{Associativity of the induced operation}
This section is devoted to the associativity of the induced operation \(T_1\star T_2\).
We first prove a general sufficient condition and then examine two complementary examples which show that the condition is meaningful but not necessary.

The proof of the main theorem relies on the following approximation lemma, which allows us to pass from the pointwise formula for \(T_1\star T_2\) to the associativity identity.

\begin{lem}\label{lem A}
Let $\To$ be a continuous t-norm on $[0,1]$. For all $f,g\in\I$ and all $x,y\in[0,1]$ with $x<y$, there exist $u_x,v_x\in(0,1]$ such that
$x<\To(u_x,v_x)$ 
and $\Tt(f(u_x^{-}),g(v_x^{-}))\ge\Tot(f,g)(y)$.
\end{lem}

\begin{proof}
If $y=1$, then $\Tot(f,g)(y)=0$, and the conclusion follows by taking $u_x=v_x=1$. Henceforth assume $y<1$.

By Proposition \ref{tens B}, set $A:=\Tot(f,g)(y)=\sup_{y<\To(u,v)}\Tt(f(u^{-}),g(v^{-}))$. Then there is a sequence of pairs $(u_n,v_n)$ with $y<\To(u_n,v_n)$ for all $n$ such that 
$$
\Tt(f(u_n^{-}),g(v_n^{-}))\ra A\quad(n\ra\infty).
$$
Since $[0,1]^2$ is compact, after passing to a subsequence if necessary, we may assume  that $u_n\ra u_0$ and $v_n\ra v_0$.  Since $y<\To(u_n,v_n)\leq\min\{u_n,v_n\}$ for all $n$, it follows that
$u_0,v_0\in[y,1]$.

By the continuity of $\To$,
$$
\To(u_0,v_0)=\lim_{n\ra\infty}\To(u_n,v_n)\geq y>x.
$$

The continuity of $\To$ guarantees the existence of a small $0<\delta<\min\{u_0,v_0\}$ such that
$$
\To(u_0-\delta,v_0-\delta)>x.
$$
Set $u_x:=u_0-\delta$ and $v_x:=v_0-\delta$. For all sufficiently large $n$, we have $u_n>u_x$ and $v_n>v_x$. Because $f$ and $g$ are non-increasing,   $f(u_n^{-})\leq f(u_x^{-})$ and $g(v_n^{-})\leq g(v_x^{-})$. Since $\Tt$ is non-decreasing in each variable,
$$
\Tt(f(u_n^{-}),g(v_n^{-}))\le \Tt(f(u_x^{-}),g(v_x^{-})).
$$
Passing to the limit as $n\ra\infty$ on the left side yields $A\le \Tt(f(u_x^{-}),g(v_x^{-}))$, completing the proof.
\end{proof}

\begin{thm}\label{thm:continuous-component}
Suppose that at least one of $\To$ and $\Tt$ is a continuous t-norm on $[0,1]$. Then $\Tot$ is a t-norm on $\I$. 
\end{thm}

\begin{proof}
By Proposition \ref{tens C}, we need to show that $\Tot$ is associative. By Proposition \ref{tens A}, without loss of generality, we assume that $\To$ is continuous. 

First, we show that for all $f,g,h\in\I$ and $x\in[0,1]$,  
$$
\Tot\big(\Tot(f,g),h\big)(x)=\sup\{\Tt\bigl(\Tt(f(u^{-}),g(v^{-})),h(w^{-})\bigr)\mid{\To(\To(u,v),w)>x}\}.
$$
On one hand, take any $u,v,w\in[0,1]$ with $x<\To(\To(u,v),w)$. Because the map $\To(-,w)$ is continuous and $x < \To(\To(u,v),w)$, there exists  some $r\in[0,1]$ such that
$$r<\To(u,v) \text{ and } x<\To(r,w).$$
It follows that 
$$
\Tt(f(u^{-}),g(v^{-}))\le \sup_{r<\To(s,t)}\Tt(f(s^-),g(t^-))=\Tot(f,g)(r)\le\Tot(f,g)(r^{-})
$$
and 
\begin{align*}
\Tt\bigl(\Tt(f(u^{-}),g(v^{-})),h(w^{-})\bigr)
&\le \Tt\bigl(\Tot(f,g)(r^{-}),h(w^{-})\bigr)\\
&\le\sup_{x<\To(s,t)}\Tt(\Tot(f,g)(s^-),h(t^-))\\
&=\Tot(\Tot(f,g),h)(x).
\end{align*}
Taking the supremum over all admissible triples $(u,v,w)$ yields $$
\sup\{\Tt\bigl(\Tt(f(u^{-}),g(v^{-})),h(w^{-})\bigr)\mid{\To(\To(u,v),w)>x}\}\leq \Tot\big(\Tot(f,g),h\big)(x).
$$
On the other hand, we prove the inequality in the opposite direction, i.e.  $$\sup_{x<\To(r,w)}\Tt(\Tot(f,g)(r^-),h(w^-))\leq\sup\{\Tt\bigl(\Tt(f(u^{-}),g(v^{-})),h(w^{-})\bigr)\mid{\To(\To(u,v),w)>x}\}.
$$ It is sufficient to show that  for any $r,w\in[0,1]$ satisfying $x<\To(r,w)$,   there exist $u_r,v_r$ with $x<\To(\To(u_r,v_r),w)$ such that 
$$\Tot(f,g)(r^-)\leq\Tt(f(u_r^-),g(v_r^-)).$$ 
By the continuity of $\To$, there are real numbers $s_0,s_1\in[0,1]$ such that
$$
r>s_0>s_1 \text{ and } \To(r,w) > \To(s_0,w) > \To(s_1,w)>x.
$$
Since $\Tot(f,g)$ is non-increasing, we have
$$
\Tot(f,g)(r^{-})=\inf_{t<r}\Tot(f,g)(t)\le\Tot(f,g)(s_0).
$$
By Lemma \ref{lem A}, there exist $u_1,v_1\in[0,1]$ such that
$$
\To(u_1,v_1)>s_1 \text{ and } \Tot(f,g)(s_0)\leq\Tt(f(u_1^{-}),g(v_1^{-})).
$$
Now take $u_r=u_1$ and $v_r=v_1$. Then we obtain 
$$
\To(\To(u_r,v_r),w)\geq \To(s_1,w)>x
$$
and $$\Tot(f,g)(r^{-})\leq\Tt(f(u_r^{-}),g(v_r^{-})) $$ as desired. 
Therefore, we have shown that 
$$\Tot\big(\Tot(f,g),h\big)(x)
=\sup\{\Tt\bigl(\Tt(f(u^{-}),g(v^{-})),h(w^{-})\bigr)\mid{\To(\To(u,v),w)>x}\}.$$

Second, the equality 
$$\Tot\big(f,\Tot(g,h)\big)(x)
=\sup\{\Tt\bigl(f(u^{-}),\Tt(g(v^{-}),h(w^{-}))\bigr)\mid{\To(u,\To(v,w))>x}\}$$
can be checked similarly. 

Because both $\To$ and $\Tt$ are associative, we obtain the equality  $$\Tot\big(\Tot(f,g),h\big)=\Tot\big(f,\Tot(g,h)\big).$$
Thus, $\Tot$ is associative as desired. 
\end{proof}

\begin{thm}\label{thm:drastic-pair}
The operation $\TD\star\TD$ is a t-norm on the fuzzy unit interval $\I$.
\end{thm}

\begin{proof}
It suffices to verify associativity. If any of $f,g,h \in \I$ equals the neutral element $\langle 1,1 \rangle$, the identity
$$
\TD\star\TD(\TD\star\TD(f,g),h) = \TD\star\TD(f,\TD\star\TD(g,h))
$$
follows directly from Proposition~\ref{tens C}.

Assume now that none of $f,g,h$ is the neutral element. We first claim that

$$\label{TD}
\TD\star\TD(f,g) = \sup\{\langle f^*(1^-),g(1^-) \rangle,\; \langle g^*(1^-),f(1^-) \rangle\}.
$$

Recall that
$$
\TD\star\TD(f,g) = \sup\bigl\{\langle \TD(u,v),\TD(a,b) \rangle \mid \langle u,a \rangle \le f,\; \langle v,b \rangle \le g\bigr\},
$$
and note that $\langle \TD(u,v),\TD(a,b) \rangle \neq \langle 0,0 \rangle$ iff both $\TD(u,v) \neq 0$ and $\TD(a,b) \neq 0$.

\begin{enumerate}[(i)]
    \item If $u = 1$, then $\langle 1,a \rangle \le f$ implies $a \le f(1^-)< 1$. From $\TD(a,b) \neq 0$ we obtain $b = 1$. The condition $\langle v,b \rangle \le g$  implies $\langle b,v\rangle \le g^\ast$, and hence \(v\le g^*(b^-)=g^*(1^-)\). Consequently,
    $$
    \langle \TD(u,v),\TD(a,b) \rangle=\lan v,a\ran\le \langle g^*(1^-),f(1^-) \rangle.
    $$

    \item If $v = 1$, a symmetric argument yields
    $$
    \langle \TD(u,v),\TD(a,b) \rangle \le \langle f^*(1^-),g(1^-) \rangle.
    $$
\end{enumerate}
Therefore,
$$
\TD\star\TD(f,g) \le \sup\bigl\{\langle f^*(1^-),g(1^-) \rangle,\; \langle g^*(1^-),f(1^-) \rangle \bigr\}.
$$

Now, we prove  the reverse inequality.  By Lemma \ref{lem:step-supdense}(1),     $\langle 1,f^*(1^-)\rangle\le f^\ast$ and $\langle 1,g(1^-) \rangle \le g $.  Besides, $\langle 1,f^*(1^-)\rangle\leq f^*$ implies  $\langle f^*(1^-),1 \rangle \leq f$. Hence,
 \begin{align*}
\langle f^*(1^-), g(1^-)\rangle&= \langle \TD( f^*(1^-),1), \TD( 1,g(1^-)) \rangle \\
&=\TD\star\TD(\langle f^*(1^-),1 \rangle, \langle 1,g(1^-)  \rangle)   \quad\text{(by \eqref{times})}\\
&\le \TD\star\TD(f,g).\end{align*}
Similarly,
$$
\langle g^*(1^-),f(1^-) \rangle \leq \TD\star\TD(g,f)= \TD\star\TD(f,g),
$$
which establishes the claim.

Since $f \neq \langle 1,1 \rangle$ and $g \neq \langle 1,1 \rangle$,  $$c = \max\{ f(1^-),g(1^-)\} < 1.$$ Similarly,  $f^\ast \neq \langle 1,1 \rangle$ and $g^\ast \neq \langle 1,1 \rangle$ imply
$$w = \max\{f^*(1^-),g^*(1^-)\} < 1.  
$$
Then by the above claim,  we have $$\TD\star\TD(f,g) \le \langle w,c \rangle.$$
For any $\langle s,t \rangle \le h < \langle 1,1 \rangle$, either $s<1$ or $t<1$, thus
$$
\TD\star\TD\bigl(\langle w,c \rangle, \langle s,t \rangle\bigr) =\langle \TD(w,s), \TD(c,t) \rangle= \langle 0,0 \rangle.
$$
Therefore,
\begin{align*}
\TD\star\TD\bigl(\TD\star\TD(f,g),h\bigr) 
&\le \TD\star\TD\bigl(\langle w,c \rangle, h\bigr)\\
&=\sup\bigl\{\langle \TD(u,s), \TD(v,t) \rangle \mid \langle u,v\rangle \leq \langle w,c\rangle, \langle s,t\rangle\leq h \bigr\}\\
&=\sup\bigl\{\langle \TD(w,s), \TD(c,t) \rangle \mid  \langle s,t\rangle\leq h \bigr\}\\
&=\langle 0,0 \rangle.
\end{align*}
Analogously, $\TD\star\TD\bigl(f,\TD\star\TD(g,h)\bigr) = \langle 0,0 \rangle$. Consequently,
$$
\TD\star\TD\bigl(\TD\star\TD(f,g),h\bigr) = \TD\star\TD\bigl(f,\TD\star\TD(g,h)\bigr),
$$
which completes the proof of associativity.
\end{proof}

The following example shows that $\Tot$ need not be associative when neither component t-norm is continuous.

\begin{exmp}\label{ex:drastic-nilpotent} Neither the t-norm $\TD$ nor $\TNM$  is continuous, although $\TNM$ is left-continuous. We show that the operation  $\DN$ is not associative.  
 
Given any $a,b\in (0,1)$ with $a<b$, let   $f:[0,1]\to [0,1]$ and $h:[0,1]\to [0,1]$ be two strictly decreasing maps that are continuous on $[0,1)$, $f(0)=h(0)=1$, $f(1^-)= h(1^-)=0.1$, $f(1)=h(1)=0$,  and $f(a)=h(b)=0.9.$  Let the map $g:[0,1]\to[0,1]$ be given by $$ g(x)=\begin{cases} 
0.9, & x\in [0,1),\\ 0, &x=1.\end{cases}$$ Clearly, $f,g,h\in \I.$ We show that $$ \DN(\DN(f,g), h)\ne \DN(f,\DN(g,h)).$$

Note that \(x<\TD(u,v)\) if and only if  \(u=1, v>x\)  or \(v=1, u>x\). Hence for all $\lam,\mu\in\I$,  
$$
(\DN)(\lam,\mu)(x)=\sup_{x<\TD(u,v)}\TNM(\lam(u^-),\mu(v^-))
=\sup\left\{\sup_{v>x}\TNM(\lam(1^-), \mu(v^-)),\;\sup_{u>x}\TNM(\lam(u^-),\mu(1^-))\right\}.
$$
To compute $F:=\DN(f,g)$,  observe that  for each \(x\in [0,1]\),
$$\sup_{v>x}\TNM(f(1^-), g(v^-))=\sup_{v>x}\TNM(0.1, g(v^-))=0,$$
 and 
$$\sup_{u>x}\TNM(f(u^-), g(1^-))=\sup_{u>x}\TNM(f(u^-), 0.9)=\begin{cases}   0.9, & x\in  [0,a),\\ f(x), & x\in [a,1],\end{cases}$$
thus
$$F(x)=
\begin{cases}
0.9, &  x\in [0,a),\\
f(x), & x\in [a,1].
\end{cases}
$$
Similarly,  $H:=\DN(g,h)$ is obtained as follows:
$$
H(x)=\begin{cases}
0.9, &  x\in [0,b),\\
h(x), & x\in [b,1].
\end{cases}
$$

Next, we compute  $\DN(F,h)$. For any \(x\in [0,1]\), $$\sup_{u>x}\TNM(F(u^-), h(1^-))=\sup_{u>x}\TNM(F(u^-), 0.1)=0,$$ and  $$\sup_{v>x}\TNM(F(1^-), h(v^-))=\sup_{v>x}\TNM(0.1, h(v^-))=\begin{cases} 0.1, & x\in[0,b),\\    0, & x\in[b,1],\end{cases}$$ 
hence,  $$\DN(F,h)(x)=\begin{cases} 0.1, & x\in[0,b),\\    0, & x\in[b,1].\end{cases}$$
Similarly,  $$\DN(f,H)(x)=\begin{cases} 0.1, & x\in[0,a),\\    0, & x\in[a,1].\end{cases}$$
Therefore,  
$$\DN(\DN(f,g), h)\ne \DN(f,\DN(g,h)).$$

\end{exmp}

\section{Conclusion}
 
We have shown that for any two t-norms $T_1, T_2$ on $[0,1]$, the operation $T_1\star T_2$ defined on the fuzzy unit interval $\I$ is always commutative, monotone and has $\langle 1,1\rangle$ as its neutral element.
Associativity--and hence the property of being a t-norm--holds in the following cases:
\begin{enumerate}
    \item At least one of $T_1, T_2$ is continuous.
    \item Both $T_1$ and $T_2$ are left-continuous; in this case $T_1\star T_2$ is precisely the tensor product of the quantales $([0,1],T_1)$ and $([0,1],T_2)$.  
\end{enumerate}
The example $T_D\star T_D$ shows that neither continuity nor left-continuity is necessary.
On the other hand, the case of $\DN$ shows failure of associativity, showing that the interaction between a left-continuous t-norm and a non-left-continuous t-norm can destroy associativity. For the reader's convenience, we summarize the known cases established or recalled in this paper in the following table.
\begin{table}[htbp] 
\centering

\begin{tabular}{llc}
\toprule
$T_1$ & $T_2$ & Is $T_1\star T_2$ a t-norm? \\
\midrule
left-continuous & left-continuous 
& yes \\

continuous & arbitrary 
& yes \\

arbitrary & continuous 
& yes \\

$\TD$ & $\TD$ 
& yes \\

$\TD$ & $\TNM$
& no \\
\bottomrule
\end{tabular}
\end{table}

The case of two left-continuous components belongs to the existing quantale tensor-product theory. The new associativity result permits one component to be arbitrary when the other is continuous; the drastic pair shows that even two components failing left-continuity may work. The contrasting example shows why these results do not extend to arbitrary pairs.

The Archimedean discussion in Section~\ref{sec3} answers a separate question: two Archimedean components imply that the fixed recursive powers of each $f<\top$ have infimum $\bot$ in $\I$. If the induced operation is associative, this gives an Archimedean t-norm.

However, a complete characterization of those pairs $(T_1,T_2)$ for which $T_1\star T_2$ is a t-norm remains an open problem.

\section*{Acknowledgements}

Rongjun Wu would like to express gratitude to the Southwest Minzu University Research Startup Funds for Grant (\#RQD2021100), which supported the writing of this article. Jie Sun would like to express gratitude to the National Natural Science Foundation of China (No. 12601905) and the Southwest Minzu University Research Startup Funds for Grant (\#RQD2024011). Hongliang Lai would like to express gratitude to the National Natural Science Foundation of China (No. 12671089).


\begin{thebibliography}{99}

\bibitem{Alsina1980} C. Alsina, E. Trillas, L. Valverde, On non-distributive logical connectives for fuzzy sets theory, \emph{Bulletin pour les sous-ensembles flous et leurs applications} 3 (1980) 18--29.

\bibitem{Alsina2006} C. Alsina, M. Frank, B. Schweizer, \emph{Associative Functions: Triangular Norms and Copulas}, World Scientific Press, Singapore, 2006.

\bibitem{Banaschewski1976} B. Banaschewski, E. Nelson, Tensor products and bimorphisms, \emph{Canadian Mathematical Bulletin} 19 (1976) 385--402.

\bibitem{DeBaets1999}B. De Baets, R. Mesiar,  Triangular norms on product lattices, \emph{Fuzzy Sets and Systems} 104 (1999), 61--75.

\bibitem{Cooman1994} G. De Cooman, E. Kerre, Order norms on bounded partially ordered sets, \emph{The Journal of Fuzzy Mathematics} 2 (1994) 281--310.

\bibitem{Dubois1980} D. Dubois, H. Prade, New results about properties and semantics of fuzzy set-theoretic operators, in: P.P. Wang, S.K. Chang (Eds.), \emph{Fuzzy Sets, Theory and Applications to Policy Analysis and Information Systems}, Plenum Press, New York, 1980, pp. 59--75.

\bibitem{Eklund2018} P. Eklund, J. Guti\'errez Garc\'ia, U. H\"ohle, J. Kortelainen, \emph{Semigroups in Complete Lattices: Quantales, Modules and Related Topics}, Developments in Mathematics, Springer, Berlin, 2018.

\bibitem{Goguen1967} J. Goguen, L-fuzzy sets, \emph{Journal of Mathematical Analysis and Applications} 18 (1967) 145--174.

\bibitem{Gutierrez2017} J. Guti\'errez Garc\'ia, U. H\"ohle, T. Kubiak, Tensor products of complete lattices and their application in constructing quantales, \emph{Fuzzy Sets and Systems} 313 (2017) 43--60.

\bibitem{Gutierrez2019} J. Guti\'errez Garc\'ia, U. H\"ohle, T. Kubiak, Some characterizations of the tensor product of complete lattices with applications to quantales, \emph{Semigroup Forum} 99 (2019) 303--316.

\bibitem{Hajek1998} P. H\'ajek, \emph{Metamathematics of Fuzzy Logic}, Springer, Dordrecht, 1998.
\bibitem{Harding2016} J. Harding, C. Walker, E. Walker, \emph{The Truth Value Algebra of Type-2 Fuzzy Sets: Order Convolutions of Functions on the Unit Interval}, CRC Press, Taylor \& Francis Group, 2016.

\bibitem{Hutton1975} B. Hutton, Normality in fuzzy topological spaces, \emph{Journal of Mathematical Analysis and Applications} 50 (1975) 74--79.

\bibitem{Klement2000} E.P. Klement, R. Mesiar, E. Pap, \emph{Triangular Norms}, Kluwer Academic Publishers, Dordrecht, 2000.

\bibitem{Lai2026} H. Lai, Q. Luo, Triangle functions generated by tensor products of quantales, \emph{Fuzzy Sets and Systems} 524 (2026) 109655.


\bibitem{Prade1980} H. Prade, Unions et intersections d'ensembles flous, \emph{Bulletin pour les sous-ensembles flous et leurs applications} 3 (1980) 58--62.


\bibitem{Schweizer1963} B. Schweizer, A. Sklar, Associative functions and abstract semigroups, \emph{Publicationes Mathematicae Debrecen} 10 (1963) 69--81.


\bibitem{Schweizer2005} B. Schweizer, A. Sklar, \emph{Probabilistic Metric Spaces}, Dover Publications, Mineola, 2005.


\bibitem{Shmuely1974} Z. Shmuely, The structure of Galois connections, \emph{Pacific Journal of Mathematics} 54 (1974) 209--225.


\bibitem{Zadeh1975a} L. Zadeh, The concept of a linguistic variable and its application to approximate reasoning I, \emph{Information Sciences} 8 (1975) 199--249.


\bibitem{Zadeh1975b} L. Zadeh, The concept of a linguistic variable and its application to approximate reasoning II, \emph{Information Sciences} 8 (1975) 301--357.
\end{thebibliography}
\end{document}